\documentclass[11pt]{article}%
\usepackage[all]{xy}
\usepackage{amssymb,latexsym}
\usepackage{graphicx}
\usepackage{amsmath,amsthm}
\usepackage{enumerate}
\usepackage{url}
\newcommand{\kk}{{{\bar k}}}

\newcommand{\F}{{\mathbb F}}

\newcommand{\Z}{{\mathbb Z}}
\newcommand{\N}{{\mathbb N}}

\newcommand{\Q}{{\mathbb Q}}
\newcommand{\Pp}{{\mathbb P}}
\newcommand{\CC}{{\mathbb C}}

\newcommand{\A}{{\mathcal A}}
\newcommand{\B}{{\mathcal B}}
\newcommand{\Cc}{{\mathcal C}}
\newcommand{\D}{{\mathcal D}}
\newcommand{\Ee}{{\mathcal E}}

\newcommand{\Ll}{{\mathcal L}}
\newcommand{\KK}{{\mathcal K}}

\newcommand{\Uc}{{\mathcal U}}
\newcommand{\V}{{\mathcal V}}
\newcommand{\Wc}{{\mathcal W}}
\newcommand{\Yc}{{\mathcal Y}}
\newcommand{\Xc}{{\mathcal X}}
\newcommand{\Zc}{{\mathcal Z}}

\newcommand{\Hm}{{\text{Hom}}}

\newcommand{\alb}{{\text{Alb}}}
\newcommand{\albdim}{{\text{Albdim}}}
\newcommand{\aut}{{\text{Aut}}}
\newcommand{\enn}{{\text{End}}}
\newcommand{\prym}{{\text{Prym}}}
\newcommand{\rk}{{\text{rk}}}

\newtheorem{thm}{Theorem}[section]
\newtheorem{theorem}[thm]{Theorem}

\newtheorem{proposition}[thm]{Proposition}
\newtheorem{assumption}[thm]{Assumption}
\theoremstyle{definition}
\newtheorem{remark}[thm]{Remark}
\newtheorem{definition}[thm]{Definition}
\newtheorem{example}[thm]{Example}

\numberwithin{equation}{section}
\begin{document}

\title{Twists of the Albanese varieties of cyclic multiple planes with  large ranks over higher dimension function fields}
\author{Sajad Salami, \\
 Inst\'{i}tuto de Matem\'{a}tica e Estat\'{i}stica \\
Universidade Estadual do Rio de Janeiro, Brazil\\
Email:  sajad.salami@ime.uerj.br}
\date{}

\maketitle

\begin{abstract}
	In  [17], we  proved a structure theorem   on the  Mordell-Weil group of  abelian varieties over function fields that  arise as the twists of abelian varieties by  the cyclic covers of  projective varieties in terms of the Prym varieties associated with covers. In this paper, we  provide an explicit way to construct
the abelian varieties with  large  ranks over  the  higher dimension  function fields. 
To do so, we apply  the above-mentioned theorem  to  the  twists of Albanese varieties of the cyclic multiple planes.
\end{abstract}

\vspace{.2cm}

\vspace{.1cm}

{\bf Subject class (2020): }{Primary 11G10, Secondary 14H40, 14H05}

\section{Introduction}
Let $\A$ be an abelian variety defined over a  given field  $k$ of characteristic $0$ or a prime $p>0$.
Denote by $\A (k)$ the set   of $k$-rational points on $\A$. It is well-known   \cite{ref7, ref11} that  $\A (k)$ is 
a  finitely generated abelain group for  the number fields as well as  the function fields under 
mild conditions. Thus, we have
$\A (k)\cong  \A (k)_{tors} \oplus \Z^r$ where 
$\A (k)_{tors}$ is a finite group called the {\it torsion subgroup}  of $\A(k)$,
and $r$ is a positive number called   the {\it  Mordell-Weil rank} or simply the {\it rank} of $\A(k)$ and denoted by $\rk (\A(k))$.
Let $\A [n](k)$ be  the group of $k$-rational $n$-division points on $\A $ for any integer $n\geq 2$.
It is remarkable  that  studying  on the  rank of abelian varieties is  more difficult than that of the torsion subgroups.  

Finding the abelian varieties of  a given dimension with large  rank  is one of the challenging  problems   in the  modern number theory. There are some interesting works in the literature depending on the ground field and the dimension of  abelian varieties.
For example,   Shafarevitch and Tate \cite{ref28} produced isotrivial elliptic curves 
with arbitrary large rank over the field $\F_p(t)$, and  a similar result was  proved by Ulmer  
in \cite{ref30} for non-isotrivial case.
In the case of higher dimension abelian varieties over $\F_p(t)$, Ulmer  
showed  \cite{ref31} that for a given prime number $p\geq 2$ and the  integers $g\geq 2$ and $ r \geq 1$ there exist absolutely simple,
non-isotrivial Jacobian varieties of dimension $g$ and rank $\geq r$ over $\F_p(t)$ for which the Birch 
and Swinnerton-Dyer's conjecture holds.
In  \cite{ref32}, Ulmer related the Mordell-Weil group of certain Jacobian varieties over $k_0(t)$, where $k_0$ is an arbitrary field, with the group of homomorphisms of other Jacobian varieties. We note that a similar result has been proved by several authors in the literature with different methods 
\cite{ref6, ref15, ref21, ref34}.
He also showed the unboundedness of the rank of  elliptic curves over 
${\bar \F}_p(t)$ by providing  an  concrete example of  elliptic curve  of large rank  with the explicit independent points.  

In  \cite{ref13, ref14}, Lapin stated  the unboundedness of the rank of  elliptic curves over  $\CC(t)$,  but his proof had a gap!   The geometric methods  of Ulmer in \cite{ref32} led to the construction of elliptic curves of moderate rank over $\CC(t)$. Moreover, his method suggested  a potential way to show the
unboundedness of  rank of the elliptic curves over $\CC(t)$, but he mentioned that  it seems likely to himself that this does not happen.

In his recent work \cite{ref33}, Ulmer proved that for a very general elliptic curve $E$ over $\CC(t)$ with height $d \geq 3$ and for  any finite rational extension $\CC(u)$  of $\CC(t)$ the  group $E(\CC(u))$ is  trivial.
It is remarkable that the largest known rank of elliptic curves over $\CC(t)$ is $68$ due to T. Shioda \cite{ref24}.
Using the  theory of the Mordell-Weil Lattices and symmetry \cite{ref23, ref25, ref26}, he also showed the existence of high rank Jacobians of curves with genus greater than one  defined over  function fields with base $\Q$.
Over the field $\CC(x,y)$, in \cite{ref16}, Libgober proved that there are  certain simple Jacobians of rank $p-1$ for any prime number $p$.

The aim of this paper is to provide  an explicit method of the construction of abelian varieties with the large ranks  over  function field of certainly defined quotient varieties of high dimensions. The main tools in this work are the theory  of twists  as well as the Albanese and Prym varieties.

Given a fixed integer $n \geq 2$,   we assume that $k$ is a  field  of 
characteristic $0$ or a prime $p > 0$ not dividing  $n$,  that contains an $n$-th root of unity  $\zeta_n$.
We denote by $\Pp^2$ the projective plane  over $\kk$, 
an algebraically closed field containing $k$.
Given  any polynomial $f(x,y) \in k[x,y]$ of degree $r  \geq 2$,
let $\Xc_n$ be the non-singular projective model of the hypersurface $X_n$ defined
by the affine equation  $ w^n =f(x,y)$.
For any integer  $m\geq 2$, we define  $\Uc_m$ to be the fibered product of $m$ copies of 
$\Xc_n$ over $k$.  Then, we let  $\V_m$ be the quotient of  $\Uc_m$ by a certain cyclic subgroup of $\aut (\Uc_m)$,
see  Section \ref{MP1} for more details.
We also denote by $\tilde{\Xc}_n$ the twist of $\Xc_n$ by the cyclic extension $\Ll | \KK$, where
$\KK=k(\V_m)$ and $\Ll=k(\Uc_m)$ are the function fields of $\Uc_m$ and $\V_m$ respectively.
We let  $\widetilde{\alb(\Xc_n)}$ be the twist of  Albanese variety $\alb(\Xc_n)$ of $\Xc_n$
by the extension $\Ll | \KK$.
We refer the reader to see Sections \ref{ALP} and \ref{twth} for the  definitions and basic properties of 
the Albanese and Prym varieties, respectively.

The main results of the paper are given as follows.
\begin{theorem}
	\label{main1}
	Notation being as above,  we  assume that there exists at least one $k$-rational point on 
	$ X_n$ and hence on   $\Xc_n$. 
	Then, as an isomorphism of  abelian groups, we have:
	$$\widetilde{\alb(\Xc_n)}(\KK) \cong \alb(\tilde{\Xc}_n)(\KK) \cong  \big (\enn_k( \alb(\Xc_n))\big)^m \oplus \alb(\Xc_n)[m](k),$$
	and hence, $\rk (\widetilde{\alb(\Xc_n)}(\KK))= m\cdot \rk (\enn_{k}(\alb(\Xc_n))),$
	where $\enn_{k}(*)$ is  the ring of endomorphisms over $k$ of its origin $(*)$
	and $\rk(\enn_{k}(*))$ denotes its rank as a $\Z$-module.
\end{theorem}

This theorem generalizes the main results of \cite{ref6,  ref21, ref34} for the higher
dimensional abelian varieties risen as the twist of Albanesse variety of  the  cyclic $n$-covers 
of a projective plane. For a given cyclic $n$-cover $\Xc_n$, by enlarging the integer $m$,  one may obtain  the abelian varieties of arbitrary large rank  over
the  function field of $\Wc_m$ with the base field $k$ of arbitrary characteristic.

On the other hand, for a fixed integer $m\geq 2$, one may be interested to compute the rank of 
$\widetilde{\alb(\Xc_n)}(\KK)$ in terms of $m$ and the other geometric quantities of  $\alb(\Xc_n)$, for example its dimension
$d_n=\dim\left( \alb(\Xc_n)\right) $. In general, $d_n\leq \dim H^0 (\Xc_n, \Omega_{\Xc_n}^1),$ where 
$\Omega_{\Xc_n}^1$ is  the vector space of differential $1$-forms on $\Xc_n$.  
We note that  $``="$ holds when  the field $k$ is of  characteristic zero. In the following we are going to calculate  $\rk (\enn_{k}(\alb(\Xc_n)))$ in terms of $d_n$ when  $k$ is a number field.

We assume that $k\subset \CC $  is a number field which contains $\Q(\zeta_n)$ and let
$\pi_n: \Xc_n \rightarrow \Pp^2$ be a cyclic $n$-cover.
Its Galois group  is a cyclic group generated by the  order $n$ automorphism $\bar{\tau}$ of $\Xc_n$
induced by 
$$\tau: (x, y, w)\mapsto (x, y, \zeta_n\cdot w),$$
an  automorphism  of $X_n$ which    acts on $\Xc_n$ and  induces an action  on  $\alb(\Xc_n)$.
Let $\alb(\Xc_n) \cong W_n/\Lambda_n$, where $W_n$ is a $\CC$-vector space of dimension $d_n$
and $\Lambda_n$ is a lattice. Then $\Lambda_n \otimes \CC$ is  of dimension $2 d_n$ as a $\CC$-vector space and we have 
$ \Lambda_n \otimes \CC = \bigoplus_\lambda^{} H_\lambda$,
where $H_\lambda$ is the eigenspace of $\displaystyle \Lambda_n \otimes \CC$ corresponding to each character $\lambda$ of $\mu_n$ the  cyclic group of primitive $n$-th roots.
Therefore, we can write $W_n=\bigoplus_\lambda^{} W_{n,\lambda}$, where  $W_{n,\lambda}=W_n \cap H_\lambda.$
To state the next results, we shall consider the following assumption.
\begin{assumption}
	\label{ass1}
	Keeping the above notations, we further assume that $\Lambda_n$ is a $\Z[\zeta_n]$-module and no conjugate character appears in $W_n$.
\end{assumption}

\begin{theorem}
	\label{main2}
	Keeping the hypothesis of Theorem \eqref{main1} and the above notation, we suppose further that  $k \subset \CC$ is a  number field  containing $\Q(\zeta_n)$ and  \eqref{ass1} holds on $\alb(\Xc_n)$  for  the positive integer $n$
	satisfying $3 \leq n \leq 12$ but $n \not = 7, 9, 11$. Then,   
	$$\rk (\widetilde{\alb(\Xc_n)}(\KK))\ge  m\cdot c_n,$$
	where $c_n=2 d_n^2$ for $n=3, 4, 5, 6, 10$ and  $c_n=  d_n^2- 8 n_1 n_2$ with $n_1+n_2=d_n/2$ for $n= 8, 12$.
	Moreover, the  equality holds  for all cases except $n=8$.
\end{theorem}
The next theorem gives a sufficient condition that provides a way to find explicit examples for the case $n=6$ in the above theorem.
\begin{theorem}
	\label{main3}
	Keeping the hypothesis of Theorem \eqref{main1},  we assume that $k \subset \CC$ is a  number field  containing $\Q(\zeta_3)$ and  \eqref{ass1} holds on $\alb(\Xc_6)$.
	Let  $F(u_0, u_1, u_2)$ be a homogeneous polynomial that can be written in  two   forms
	$F=G_1^2+ H_1^3=G_2^2+H_2^3$, where 
	$G_1, G_2, H_1, H_2 \in k[u_0, u_1,u_2]$ are homogeneous polynomials  such that   
	$$\{\lambda_0 G_1^2 + \lambda_1 H_1^3:\  [\lambda_0: \lambda_1] \in \Pp^1 \}, \ \ 
	\{\lambda_0 G_2^2 + \lambda_1 H_2^3:\  [\lambda_0: \lambda_1] \in \Pp^1 \},$$
	are two different pencils. Then,
	$	\rk (\widetilde{\alb(\Xc_6)}(\KK))= 2 m d_6^2.$
	One may consider the  following two cases  as explicit examples:  
	\begin{itemize}
		\item [(i)] $\Xc_6$ is associated to $\Cc=\Ee^\vee$ the dual  of a smooth plane cubic $\Ee$;
		\item [(ii)] $\Xc_6$ is associated to the affine curve
		$$C: f(x,y)=x^3 -3 x  y(y^3-8)+2(y^6 -20 y^3 -8)=0.$$
	\end{itemize}
\end{theorem}

\begin{remark}
	\label{rem1}
	\rm
	\begin{itemize}
		\item[(1)]  The Theorem \ref{main1} can be generalized for
		the twists of Albanese variety associated with   $n$-covers of the projective space $\Pp^\ell$ of dimension $\ell \geq 3$. Indeed, given any polynomial $f(u_1, \ldots, u_\ell)$ with coefficients in $k$, if $\Xc_n$ denotes a non-singular projective model of the hypersurface defined by the affine equation
		$w^n=f(u_1, \ldots, u_\ell)$,  and $\tilde{\Xc}_n$, $\Uc_m$, $\V_m$, $\KK$ and  $\Ll$  are defined in a similar way, then the assertion of Theorem \ref{main1} hold for $\tilde{\Xc}_n$ by adapting its proof. Furthermore, it is remarkable that the cyclic $n$-covers of the projective line $\Pp^1$ has been treated by the author in \cite{ref21}.
		
		\item[(2)] 	One can use the Theorems \ref{main2} and \ref{main3} to  obtain abelian varieties of  arbitrary large ranks over the 
		number fields by considering  the Silverman's specialization theorem \cite{ref27} and study on the $k$-rational points on the variety $\V_m$.
		
	\end{itemize}
\end{remark}

The present paper is organized as follows.
The Section  \ref{ALP} is devoted to recalling the definitions and  fundamental properties of the Albanese and Prym varieties.
In Section \ref{twth}, we  briefly review  the  basics of   twisting  theory of algebraic varieties. 
Then, we provide some  results on the multiple planes $\Xc_n$ in Section \ref{cyccov}. 
We give the proofs of Theorems \ref{main1},  \ref{main2} and \ref{main3} in Sections \ref{MP1} and \ref{MP2}.

\section{Albanese and Prym varieties}
\label{ALP}

In this section, we let $\Xc$ be a non-singular projective variety over  $\kk$ and $\text{Pic}(\Xc)$ denotes its reduced Picard variety. By an {\it abelian variety over}  $\kk$ we mean an algebraic group $\A$ over $\kk$ which is
non-singular,  proper, and connected as a variety.  In below, we are going to give the definition and basic properties of an abelian varity associated to $\Xc$. For more details, we refer the reader to \cite{ref10, ref22}.

\begin{definition}
	{\rm
		The {\it Albanese variety} of $\Xc$ is
		an abelian variety $\alb (\Xc)$ together with the  morphism of varieties 
		$\alpha: \Xc \rightarrow \alb (\Xc)$ satisfying the following universal property:
		
		{\it  For any morphism $\alpha' : \Xc \rightarrow \A $, where $\A$ is an abelian variety,  there exists a unique homomorphism  $\alpha'': \alb(\Xc) \rightarrow \A$, up to a translation,
			such that the following diagram commutes:}
		\begin{equation}\label{diag1}
		\xymatrixcolsep{2pc}
		\xymatrix{
			\Xc   \ar[r]^{\alpha}   \ar[d]^{id} & \alb(\Xc)
			\ar[d]^{\exists! \ \alpha''} \\ \Xc   \ar[r]^{\alpha'}    & \A }
		\end{equation}
		The morphism $\alpha$ is called the {\it Albanese morphism} and $\alpha(\Xc)$
		is known as the  {\it Albanese image} of $\Xc$. Moreover, 
		the dimension of $\alpha(\Xc)$ is called the {\it Albanese dimension} of $\Xc$.}
\end{definition}

The notion of  Albanese variety is a classic one in the case $\kk \cong \CC$. Indeed, for a smooth complex variety  $\Xc$ one has $\alb(\Xc)= H^0(\Xc, \Omega_\Xc^1)^\vee /H^1(\Xc, \Z)$, and the Albanese  map is given by 
$x \mapsto \int_{x_0}^x \omega$ where $x_0$ is a base point  and the integral is viewed as a linear function on 
$H^0(\Xc, \Omega_\Xc^1)$ well-defined up to the periods.

It can be shown  that if $\alb(\Xc)$ exists,  then it is unique up to isomorphism.
Moreover, it  is equal to  $\text{Pic}^{\vee} (X)$, the dual of  reduced Picard variety of $\Xc$.
For example,  when $\Xc$ is a non-singular projective curve, then the Albanese variety is nothing but the Jacobian variety of $\Xc$.

We note that the Albanese variety $\alb(\Xc)$ is generated by  $\alpha(\Xc)$, i.e,
there is  no abelian subvariety of $\alb(\Xc)$ containing $\alpha(\Xc)$. In particular,  $\alpha(\Xc)$ is not reduced to a point  if $\alb(\Xc)$ is not a singleton set.  In general, one has
$\dim \alb(\Xc) \leq   \dim H^0(\Xc, \Omega_\Xc^1)$, where  $\Omega_\Xc^1$  is the vector space 
of regular $1$-forms on $\Xc$ and  the equality holds if $\kk$ is of characteristic zero.
We note that the quantity $\dim H^0(\Xc, \Omega_\Xc^1)$ is well-known as  the {\it irregularity} of  $\Xc$.

The assignment $\Xc \mapsto \alb(\Xc)$  is a covariant functor 
from the category of  non-singular projective varieties to
the category of  abelian varieties. 
In other words,  for any morphism $\pi : \Xc' \rightarrow \Xc$ of the non-singular projective varieties $\Xc$ and $ \Xc'$, 
there exists a unique morphism $\tilde{\pi} : \alb(\Xc') \rightarrow \alb(\Xc) $ such that the  digram (\ref{diag2}) in below  commutes.
In particular,  if the morphism $\pi$   is  surjective, then the morphism 
$\tilde{\pi}$ is also surjective.

\begin{equation}\label{diag2}
\xymatrixcolsep{2pc}
\xymatrix{
	\Xc'   \ar[r]^{\pi }   \ar[d]^{\alpha'} & \Xc \ar[d]^{\alpha} \\
	\alb(\Xc')  \ar[r]^{\tilde{\pi} }    &  \alb(\Xc) }
\end{equation}

In the rest of this section, we recall the definition of Prym varieties for cyclic $n$-cover
$\pi : \Xc'   \rightarrow \Xc$ of  non-singular projective varieties.  
Classically,
the notion  of Prym variety has been introduced  by Mumford \cite{ref19} for the double covers of curves over the  complexes  and it is   extensively studied by Beauville in  \cite{ref1}. Then,  it has been considered for the  double covers of smooth surfaces by Khashin in \cite{ref8}, and for the double covers of projective varieties 
by  Hazama in \cite{ref6}. Recently,  the Prym variety  of  cyclic $n$-covers  of curves are studied in \cite{ref12}.
In   \cite{ref18},   it is defined in a more abstract setting,  but is used for arbitrary covers of curves over the complexes.

\begin{definition}
	{\rm  The {\it Prym variety} of the cyclic  $n$-covers   $\pi : \Xc' \rightarrow \Xc$  of non-singular projective varieties
		is defined as the  quotient  abelian variety, 
		$$\prym_{\Xc'/\Xc}:=\frac{\alb(\Xc')}{\text{Im} (id+\tilde{\gamma} +\cdots +\tilde{\gamma}^{n-1})},$$
		where  $\tilde{\gamma}$ is the automorphism on  $\alb(\Xc')$
		induced by an order $n$ automorphism  $\gamma $ of $ \Xc'$.}
\end{definition}


When   $\Xc, \Xc'$ are irreducible and  both of them as well as  the cyclic $n$-cover  $\pi$  are defined over $k$,   there is a $k$-isogeny of abelian varieties,
\begin{equation}
\label{eq1}
\prym_{\Xc'/\Xc} \sim_k \ker(id + \tilde{\gamma} + \cdots + \tilde{\gamma}^{n-1} :  \alb(\Xc') \rightarrow \alb(\Xc'))^\circ,
\end{equation}
where $(*)^\circ$ means the unique connected component of the origin $(*)$.

It is  useful to construct the  new  cyclic $n$-covers using the  given ones. 
To this end, we let   $\pi_i: \Xc'_i \rightarrow \Xc_i$  $(i=1,2)$ be  two cyclic $n$-covers of  irreducible non-singular projective varieties, 
$\gamma_i \in \aut(\Xc'_i)$ be an  automorphism of order $n\geq 2$ for $i=1,2$, all defined over $k$.
Moreover, we assume that there exist $k$-rational points $x'_i \in \Xc'_i(k)$ for $i=1,2$.
Then, we have  a $k$-rational isogeny of abelian varieties,
\begin{equation}
\label{eq2}
\prym_{\Xc'_1 \times \Xc'_2/\Yc} \sim_k  \prym_{\Xc'_1/\Xc_1}\times  \prym_{ \Xc'_2/\Xc_2},
\end{equation}
where $\Yc=\Xc'_1 \times \Xc'_2/G$  is   the intermediate cover and 
$G$ is the cyclic group generated by $\gamma= (\gamma_1, \gamma_2) \in \aut(\Xc'_1 \times \Xc'_2) $.

For the proof of above assertions, we refer the reader to \cite{ref21}.
Note that  we have to restrict ourselves to non-singular projective varieties in Section 2 of \cite{ref21} instead of the  quasi-projective ones.
Because,  $\alb(\Xc)$ is a semi-abelian variety for any  quasi-projective variety $\Xc$.
This means that it is an extension of an algebraic group by a torus. 
For more details, consult \cite{ref22}.

\section{$G$-sets and twists}
\label{twth}
In this section, we briefly recall the two equivalent  definitions of  the twist of an algebraic variety 
and its basic properties.  To see more on the subject, consult  \cite{ref2}.

Let $\KK$ be a field  and $\Ll|\KK$ a finite extension with Galois group $G=\text{Gal}(\Ll|\LARGE\KK)$. 
A {\it $G$-set} is a discrete topological space $\Xc$ such that the left action of $G$ on $\Xc$ is continuous. 
For every $x \in \Xc$ and $u \in G$, we denote by $^u x$ the left action of  $u$ on $x$. 
A {\it $G$-group} is a $G$-set  $\A$ equipped with a group  structure 
invariant under the action of $G$, i.e., $ {}^{u}(x\cdot y)={}^{u}x \cdot {}^{u}y$ for  $ x, y \in \A$ and $u \in G$.
Any continuous application $a: u \mapsto a_u$ of $G$ to a $G$-set $\A$ is called a {\it cochain} of $G$ with values in $\A$.
A cochain $a=(a_u)$ is called a {\it $1$-cocycle} of $G$ with  values in $\A$ if  $a_{uv}= a_u \cdot {}^{u}a_v$ for  $u, v \in G$.
For any $1$-cocycle $a=(a_u)$, one has $a_{id} =1 $ and $a_u \cdot {}^{u} a_{u^{-1}} =1$, where $u \in G,$ and 
$1\in \A$ is  the identity element.
The set of $1$-cocycles of $G$ with values in a $G$-set $\A$, is denoted by $\Zc^1(G,\A)$. 
We say that a {\it  $G$-group $\A$ acts on the $G$-set $\Xc$   from left}, in a compatible way with the action of $G$, if there is an application
$(a,x) \rightarrow a\cdot x$ of $\A \times \Xc$ to $\Xc$ satisfying the following conditions:
\begin{itemize}
	\item [(i)] ${}^{u}(a\cdot x)= {}^{u}a \cdot {}^{u}x $ \ \  $(a \in \A,  x \in \Xc, u\in G )$
	\item [(ii)] $a\cdot (b \cdot x) = (a \cdot b) \cdot x, \text{and} \  1 \cdot x=x, \ $ 
	$(a, b\in \A,  x \in \Xc).$
\end{itemize}
Let $\A$ be a $G$-group, $a=(a_u) \in \Zc^1(G, \A)$  a $1$-cocycle of $\A$,
and  $\Xc$  a $G$-set that is compatible with the group action of $G$. 
For any $u \in G$ and $x\in \Xc$, define ${}^{u'}x:= a_u \cdot {}^{u}x$. 
The $G$-set with this action of $G$ is denoted by $\Xc_a$ and is called the {\it twist  of $\Xc$} obtained  by the $1$-cocycle $a=(a_u)$. 

In what follows, we will give another definition of the twist in terms of  the schemes and its relation with the previous  definition.
Let $\Xc$ be a projective scheme defined over $k$ and denote its function field by  $\KK$.
Let $\aut(\Xc)$ be the automorphism scheme of $\Xc$ and $a=(a_u) \in \Zc^1(G,\aut(\Xc))$ be a $1$-cocyle. 
Then,  $\alb(a):= (\alb(a_u))$ satisfies the $1$-cocycle condition, i.e.,
$\alb(a)  \in \Zc^1(G,\aut(\alb(\Xc))).$
Indeed, the equality  $a_{uv}=a_u \circ \ ^ua_v$  implies 
$\alb(a_{uv})=\alb(a_u) \circ {}^{u} \alb(a_v)$, 
since the construction of the  Albanese variety is compatible with base change. 
Here, we have a proposition that provides the  second  definition for the twist of $\Xc$.
For its proof, one can see the Propositions 2.6 and 2.7 in \cite{ref2}.
\begin{proposition}
	\label{twotw}
	Keeping the above notations, there exist a unique quasi-projective $\KK$-scheme $\tilde{\Xc}$  and a unique $\Ll$-isomorphism 
	$$g: \Xc \otimes_{\KK} \Ll \rightarrow \tilde{\Xc} \otimes_{\KK} \Ll$$
	such that ${}^{u}g=g\circ a_u$ holds for  any $u \in G$.
	The map $g$ induces an isomorphism of the twisted $G$-set $\Xc_a(\Ll)$ onto the  $G$-set 
	$\tilde{\Xc}(\Ll)$. 
\end{proposition}
The scheme $\tilde{\Xc}$ in the above theorem  is called the 
{\it twist  of $\Xc$} by the extension $\Ll | \KK$, or equivalently by the $1$-cocycle $a=(a_u)$.

The following proposition shows the  relation between  Albanese variety and the twists.
\begin{proposition}
	\label{P1}
	Keeping the above notations, the twist $\alb(\Xc)_{\alb(a)}$ of $\alb(\Xc)$ by the $1$-cocycle $\alb(a)$ is $\KK$-isomorphic to $\alb(\Xc_a)$.
	Equivalently, $\widetilde{\alb(\Xc)}$ the twist of $\alb(\Xc)$   is $\KK$-isomorphic to $\alb(\tilde{\Xc})$.
\end{proposition}
\begin{proof}
	If $g: \Xc \otimes_{\KK} \Ll \rightarrow \Xc_a \otimes_{\KK} \Ll$ denotes the isomorphism 
	such that $^ug=g \circ a_u$, then the induced isomorphism of Albanese varieties
	$$\alb(g): \alb (\Xc)\otimes_{\KK} \Ll \rightarrow \alb(\Xc_a) \otimes_{\KK} \Ll$$ satisfies 
	$^u\alb(g)=\alb(g)\circ \ ^u \alb(a_u)$  for any $u \in G$, by the functoriality of twists.
	Therefore,  the uniqueness of the twist implies that  $\alb(\Xc)_{\alb(a)}$ is $\KK$-isomorphic to $\alb(\Xc_a)$. Equivalently,  we have 
	$\widetilde{\alb(\Xc)} \cong_{\KK}\alb(\tilde{\Xc})$.
\end{proof}

In  \cite{ref21}, we proved a structure theorem on the set of  rational points of the twists of  abelian varieties by cyclic covers  over function fields of (quasi)-projective varieties.
In order  to make this work be self-contained and for convenience of the reader,
we recall  here the main result of \cite{ref21} which plays an essential rule in the proof of Theorem \ref{main1}.

Given an integer $n \geq 2$, let  $\pi : \Xc' \rightarrow \Xc$  be a cyclic $n$-cover of irreducible projective varieties,  both as well as $\pi$ defined over a field $k$.
Denote by  $\KK$ and   $\Ll$  the function fields of $\Xc$ and $\Xc'$ respectively.
Assume that $\A$ is an abelian variety with an automorphism $\sigma$ of order $n$, and 
let $\A[n](k)$ be the  group of $k$-rational $n$-division points on $\A$.
Define  $\tilde{\A}$ to be the twist of $\A$ by the cyclic extension $\Ll|\KK$,
or equivalently, 
by  the  $1$-cocycle $a=(a_u) \in \Zc^1(G, \aut(\A))$, where $a_{id}=id$,  $a_{\gamma^j}=\sigma^j$ for 
$j=0, \ldots, n-1,$  and $G$ is the Galois group of the extension $\Ll|\KK$. 
The  following theorem  describes the structure of  Mordell-Weil  group of $\KK$-rational points  on the twist $\tilde{\A}$ which  generalizes the main result of \cite{ref6, ref34}. 

\begin{theorem}  
	Notation being as above, we assume that there exists at least one  $k$-rational  
	point  on $ \Xc'.$  	Then, 
	$$\tilde{\A}(\KK) \cong \Hm_{k} (\prym_{\Xc'/\Xc}, \A) \oplus \A [n](k).$$
	as an isomorphism of  abelian groups.
	Moreover, if  $\prym_{\Xc'/\Xc}$ is $k$-isogenous with $ \A^m \times \B$ 
	for some integer $m>0$ and $\B$ is an   
	abelian variety defined over $k$ so that 
	none of its  irreducible components is $k$-isogenous to $\A$,  then 
	$$\rk(\tilde{\A}(\KK))= m\cdot \rk (\enn_{k}(\A)).$$
\end{theorem}

We used this opportunity to correct Theorem 1.1 in \cite{ref21} as the above form by removing the irrelevant conditions  on the dimension of $\B$ as well as changing the symbol ``$\geq$" with ``$=$" in the assertion of the theorem.
Indeed, by rechecking and specially considering the
series of  isogenies in the proof of above theorem, one can conclude that
the dimension of   $\B$ does not have any role  in the proof  and we have ``$=$" instead of ``$\geq$".

\section{Cyclic multiple planes }
\label{cyccov}
Given an integer $ n\geq 2$ and a polynomial $f(x,y)\in k[x,y]$ of degree $r\geq 2$,
we suppose that $F(u_0, u_1,u_2)$  is a homogeneous  polynomial such that $f(x,y)=F(1,x,y)$.
Denote by $C$ and $\Cc$ the affine and projective plane curves defined  by $f=0$ and  $F=0$, respectively.
Let  $L_\infty$ be  the line at infinity, say $u_0=0$. 
Assume that  $e$ is the smallest integer satisfying $e \geq n/r$ and set $n_0=ne-r$.
Define $X_n$ to be the projective surface given by the affine equation $w^n=f(x,y)$, which
can be expressed by the equation  $u_3^n=u_0^{n_0} F(u_0,u_1,u_2)$ in the weighted projective space $\Pp^3_{(1,1,1,n_0)}$.
Let $B$ be the branch locus of the map  $p: X_n\rightarrow \Pp^2$ which  drops the last coordinate.
Then $B$ is $\Cc$ if $n_0=0$ and $\Cc \cup L_\infty$ otherwise.
Doing a series of blow-ups  gives us a map $\psi: Y\rightarrow \Pp^2$ so that  $B':=\psi^{-1}(B)$ has normal crossings.
Hence, the projection on the second factor  $p': X'_n=X_n\times Y \rightarrow Y $ will be  a cyclic $n$-cover of $Y$. 
Let $\nu:  X''_n \rightarrow X'_n$ be the normalization map and $p'': X''_n\rightarrow Y $ the composition of $\nu$ and $p'$, which is again a cyclic $n$-cover of $Y$.
Denote by $\Xc_n$ the desingularization  of $X''_n$ and  let 
$\beta: \Xc_n \rightarrow X''_n$ be a morphism  such that
the map $\tilde{p}= p'' \circ \beta: \Xc_n \rightarrow Y$ is a cyclic $n$-cover of $Y$ over an open subset of $Y$
and the Albanese map $\alpha: \Xc_n \rightarrow \alb(\Xc_n)$  factor through a map
$\alpha': X''_n \rightarrow \alb(\Xc_n)$.
Thus, we obtain a cyclic $n$-cover $\pi_n: \Xc_n \rightarrow \Pp^2$  which is the composition 
$ \psi \circ \tilde{p} $ and  fits in the following commutative diagram.

\begin{equation}  
\label{diag3}
\xymatrixcolsep{2pc}
\xymatrix{ 
	& \ar[dl]_{\nu} X''_n \ar[r]^{\alpha'} \ar[d]^{p''} & \alb(\Xc_n)  \\
	X'_n \ar[r]^{p'}  \ar[d]    &    Y \ar[d]^{\psi}  & \ar[l]^{\tilde{p}} \ar[ul]_{\beta} \ar[dl]^{\pi_n} \Xc_n \ar[u]_{\alpha}\\
	X_n  \ar[r]_{p} &  \Pp^2 &  }
\end{equation}

Recall that $d_n$ is the dimension of  Albanese variety  $\alb(\Xc_n)$.
It is a well known fact that $d_n \leq \dim H^0(\Xc_n, \Omega_{\Xc_n})$ 
and $``="$ holds when  the base field $k$ is of characteristic zero.
If we  assume that  the polynomial  $f(x,y)$ has a decomposition $f=f_1^{m_1} \cdots f_d^{m_d}$ into the irreducible factors over $\kk\cong \CC$ such that $\gcd(n_0, m_1, \ldots, m_d)=1$, which implies  the irreducibility of  $\Xc_n$, then
letting $r_i=\deg(f_i)$ for $1\leq i \leq d$, we have
\begin{equation}
\label{ineq}
0 \leq  d_n \leq 
\begin{cases}
\displaystyle
\frac{1}{2}(n-1)\big(\sum_{i=1}^d r_i-2\big), & \text{if } n|r\\
\displaystyle
\frac{1}{2}(n-1)\big( \sum_{i=1}^d r_i-1\big), & \text{otherwise,}
\end{cases}
\end{equation}
One can  see Proposition 1 and its corollary in \cite{ref20} for the proof of above inequality.

Here, there exists an example for which $d_n=\dim H^0(\Xc_n, \Omega_{\Xc_n}) >0$.
\begin{example}
	{\rm 
		Denote by $\D_\ell$ an $\ell$-cyclic covering of the projective line
		$\Pp^1$ defined  by  the equation $v_2^\ell=\prod_{i=1}^t (b_i v_0- a_i v_1)^{s_i}$ with
		$\ell | (s_1+\cdots + s_t)$.  Let $\varphi: \Pp^2\rightarrow \Pp^1$ be a rational map given by
		$$\varphi((u_0:u_1:u_2))=(F_1(u_0,u_1,u_2):F_2(u_0,u_1,u_2)),$$
		where both of $F_1$ and $F_2$  are  homogeneous polynomials of degree $s$. 
		If $n $ divides $s \cdot (s_1+\cdots + s_t)$ and $\ell$ is a divisor of $n$, then
		the cyclic multiple plane $\Xc_n$
		associated with $X_n$ defined  by the affine equation,
		$$ w^n=f(x,y)=\prod_{i=1}^t (b_i F_1(1,x,y)- a_i F_2(1,x,y))^{s_i},$$
		factors through  $\D_\ell$.
		In this case, it is said that $\Xc_n$ {\it factors through a pencil}.
		One can see that if $\D_\ell$  is a curve of   genus $\geq 1$, then $\dim H^0(\Xc_n, \Omega_{\Xc_n})>0$.}
\end{example}

From now on, we assume that $k \subset \CC$  is a number field that contains an $n$-th root of unity denoted by $\zeta_n$.
In order to give a description of the structure of  $\alb(\Xc_n)$ over $\CC$,
we also suppose that the assumption \eqref{ass1} holds for $\alb(\Xc_n)$ for  $2\leq n\leq 12$ but $n \neq 7,9,11$.
Let  $\Ee_i$ and $\Ee_\rho$ be  the elliptic curves
associated  to the lattices  $\Lambda_\rho=\Z \oplus \rho \Z$ and $\Lambda_i=\Z \oplus i \Z$, where $\rho=\zeta_3$ and  $i=\zeta_4$. It is easy to see that the affine  Weierstrass forms of the
elliptic curves $\Ee_i$ and $\Ee_\rho$ are 
$y^2=x^3+x$ and $y^2=x^3+1$, respectively.
We denote by  $C_1$ and $C_2$  the genus $2$ curves which are  the normalization of the projective closure of the affine curves $y^2=x^5+1$ and $y^2=x^5+x$ and let $J(C_1)$ and $J(C_2)$ be  their Jacobians varieties, respectively.

\begin{proposition}
	\label{p2}
	Keeping the above notations and assumptions, for a given 
	$n \in \{ 2,3,4,5,6,8,10,12\}$,  we have 
	
	\begin{itemize}
		\item [(i)]   $\Xc_2$ factors through a pencil; 
		\item [(ii)] For $n=3,6 $, if $\alpha(\Xc_n)$ is a surface, then
		$\alb(\Xc_n)\cong \Ee^{d_n}_\rho$;
		\item [(iii)]  For $n=4$, either $\Xc_4$ factors through a pencil, or 
		$\alb(\Xc_4)\cong \Ee^{d_4}_i$.
	\end{itemize}
	For $n=5,8,10,$ and $12$,  the dimension of $\alb(\Xc)$ is an even  integer,    and 
	\begin{itemize}
		\item [(iv)]  $\alb(\Xc_n)\cong J(C_1)^{d_n/2}$ for $n=5, 10$;
		\item [(v)]  $\alb(\Xc_8)\cong J(C_2)^{n_1} \times \Ee_i^{2n_2}$ with $n_1+n_2=d_8/2$;  
		\item [(vi)]  $\alb(\Xc_{12})\cong \Ee_i^{2n_1} \times \Ee_\rho^{2n_2}$ with $n_1+n_2=d_{12}/2$. 
	\end{itemize}
\end{proposition}
\begin{proof}
	The part (i) is a result of de Franchis \cite{ref5}.
	When $\alpha(\Xc_n)$ is a surface, the part (ii)  for $n=3$ is due to Comessati in \cite{ref4}, which is  proved in Theorem 4.10 of \cite{ref3} with a different method; 	Otherwise, it is a consequence of 
	Theorem 5.7 and the remarks (5.5) and (5.6) in \cite{ref3}.
	The part (iii) is Theorem 4 of \cite{ref20} or one can conclude it from Theorem 5.7 of \cite{ref3} again.
	All of the assertions in (iv-vi) are consequences of Theorem 5.8 of \cite{ref3}, when $\alpha(\Xc_n)$ is a surface; and they can be concluded in general by  Theorem 5.9 and remark (5.11) of \cite{ref3}.
\end{proof}

We recall that the {\it Albanese dimension} of  the affine curve 
$$C: f(x,y)=F(1,x,y)=0$$ is defined as
$\albdim(C)=\max_{n\in \N} \dim \alpha (\Xc_n)$
where $\alpha_n: \Xc_n \rightarrow \alb(\Xc_n)$ is the Albanese map.
A priori, $\albdim(C)$ can take the values $0,1,$ or $2$.
In   Theorem  $1$ (ii) of  \cite{ref9}, Kulikov provides a sufficient condition to have $\albdim(C)=2$ without giving any concrete example.

\begin{theorem}
	\label{kuliko}
	Assume that $F(u_0, u_1, u_2)$ can be written in  two different  forms  $F=G_1^a+ H_1^b=G_2^a+H_2^b$, where
	$a, b$ are co-prime integers $\geq 2$ and $G_1, G_2, H_1, H_2 \in k[u_0, u_1,u_2]$, such that
	the	two pencils 
	$$\{\lambda_0 G_1^a + \lambda_1 H_1^b:\  [\lambda_0: \lambda_1] \in \Pp^1 \}, \ \ 
	\{\lambda_0 G_2^a + \lambda_1 H_2^b:\  [\lambda_0: \lambda_1] \in \Pp^1 \},$$
	are different over $\CC$.
	Then $\alpha(\Xc_{ab})$ is a surface and hence $\albdim(C)=2$ for
	$C: f(x,y)=F(1,x,y)=0.$
\end{theorem}

In Theorem 0.2 of \cite{ref29}, Tokunaga demonstrated an explicit irreducible affine plane
curve satisfying the Kulikov's condition as follows.
\begin{proposition}
	\label{p3}
	The Albanese image $\alpha(\Xc_6)$ is a surface  in the following two cases:
	\begin{itemize}
		\item [(i)] $\Xc_6$ is   associated to $\Cc=\Ee^\vee$ the dual  of a smooth plane cubic $\Ee$;
		\item [(ii)] $\Xc_6$ is  associated to the projective plane sextic curve $\Cc$ defined by 
		$$ F(u_0,u_1,u_2)= u_0^3u_1^3 -3 u_0 u_1 u_2( u_2^3-8)+ 2(u_2^6+ 20 u_0^3 u_2^3 -8 u_0^6)=0.$$
	\end{itemize}
	Moreover, we have $\albdim(C)=2$ for $C$ the affine model of $\Cc$ in both cases.
\end{proposition}

\section{Proof of Theorem \ref{main1}}
\label{MP1}
It is   clear that the hypersurface $X_n$ defined by $w^n=f(x,y)$ admits an order $n$ automorphism
$ \tau: (x,y,w) \mapsto ( x, y, \zeta_n \cdot w),$ 
which induces an automorphism on
its non-singular projective model  $\Xc_n$ denoted by $\bar{\tau}.$
For any  integer $m\geq 1$, we define $U_m:=X_n^{(1)} \times_k \cdots \times_k X_n^{(m)}$ 
where $X_n^{(i)}$ is  a copy of $X_n$ given by the affine equation
$w_i^n=f(x_i,y_i)$ for   $1 \leq i \leq m$. Denote  by $\tau_i$  the corresponding automorphism.
Then $\gamma=(\tau_1, \ldots, \tau_m)$ is an order $n$ automorphism of $U_m$ which naturally induces  an  order $n$ automorphism $\bar{\gamma}=(\bar{\tau}_1, \ldots, \bar{\tau}_m) $ of the fibered product 
$\Uc_m=\Xc_n^{(1)} \times_k \cdots \times_k \Xc_n^{(m)}$, where $\Xc_n^{(i)}$ is a non-singular 
projective model of $X_n^{(i)}$ for 
each $1\leq i \leq m$. Note that $\Uc_m$  can be viewed as a non-singular model of $U_m$.
Denote by $L$ and $\Ll$ the function fields of $U_m$ and $\Uc_m$, respectively.
Let $\aut(*)$ be the automorphism group of its origin $(*)$ and 
$G= \left\langle \gamma\right\rangle$ and  $\bar{G}= \left\langle \bar{\gamma} \right\rangle $  be the cyclic subgroup of $\aut (U_m)$ and   $ \aut (\Uc_m)$ generated by $ \gamma$ and $\bar{ \gamma}$, respectively. 
Let $V_m$ and $\V_m$ be the quotient of $U_m$ and $\Uc_m$ by $G$ and $\bar{G}$,
and denote by $K$ and $\KK$ the function fields of $V_m$ and $\V_m$, respectively.
Then, both of the extensions $L|K$ and $\Ll| \KK$ are finite cyclic extension of order $n$.
Indeed, we have $L\subset k(x_1, x_2, \ldots ,x_m, y_1, y_2, \ldots,   y_m, w_1, \cdots, w_m),$ 
where  $x_i$'s and $y_i$'s are independent transcendental  variables and each $w_i$ 
satisfies in the following equations,
\begin{equation}
\label{ho0}
w_i^n-f(x_i,y_i)=0 \ (i=1,\ldots, m).
\end{equation}
Then, the field $K$ is  the $G$-invariant elements of $L$, i.e., 
$$K=L^G\subseteq  k ( x_1, \ldots, x_m, y_1, \ldots, y_m, w_1^{n-1} w_2 , \ldots, w_1^{n-1}w_{m-1}).$$
Since $(w_1^{n-1} w_{i+1})^n=f(x_1, y_1)^{n-1}f(x_{i+1}, y_{i+1})$  for $1 \leq i \leq  m-1$, so by defining  $z_i:=w_1^{n-1}w_{i+1}$   the variety  $V_m$ can be expressed by the equations 
\begin{equation}
\label{ho1}
z_i^n=f(x_1,y_1)^{n-1}f(x_{i+1}, y_{i+1} ) \ (i=1,\ldots, m-1).
\end{equation}
Thus, $L|K$ is  a cyclic   extension  of degree $n$ determined by the equation
$w_1^n=f(x_1,y_1)$, i.e., 
$$L=K(w_1)\subseteq k(x_1, \ldots, x_m, y_1, \ldots, y_m, z_1, \ldots, z_{m-1})(w_1).$$
Note that the cyclic field extension $\Ll|\KK$ can be determined by considering the homogenization 
of the equations (\ref{ho0}). Furthermore, the variety $\V_m$ can be expressed by the homogenization of equations (\ref{ho1}).

Let  $\tilde{X}_n$   and $\tilde{\Xc}_n$   be the twists  of $X_n$  and  $\Xc_n$   by the extensions $L|K$ and $\Ll|\KK$, respectively.
Then, one can check that the twist $\tilde{X}_n$ is given by the affine equation,
\begin{equation}
\label{ho2}
f(x_1,y_1)  z^n=f(x,y), 
\end{equation}
and  $\tilde{\Xc}_n$ can be determined by its homogenization.
Moreover, it is  easy to  check that $\tilde{X}_n$ contains the $m$   $K$-rational points:
\begin{equation}
\label{Kpoint}  
P_1:=\left( x_1, y_1, 1\right)\ 
\text{and} \ P_{i+1}:=( x_{i+1},  y_{i+1},  w_{i+1}/w_1) \ \text{ for}  \  1\leq i\leq  m-1.
\end{equation}
Let us  denote by $\tilde{P}_i$ the point corresponding to $P_i$ on $\tilde{\Xc}_n$ for $i=1, \ldots, m$.
Now, by  applying    (\ref{eq1}) to  the  $n$-cover $\pi: \Xc_n\rightarrow \Pp^2$, we obtain 
$$\prym_{\Xc_n^{(i)}/\Pp^2} =\frac{\alb(\Xc_n^{(i)})}{\text{Im} (id+ \tilde{\gamma}+ \cdots +\tilde{\gamma}^{n-1})}
\sim_k \ker \big(id+ \tilde{\gamma}+ \cdots +\tilde{\gamma}^{n-1})^\circ.$$
Since $0=id - \tilde{\gamma}^n=(id-\tilde{\gamma})(id+ \tilde{\gamma}+ \cdots +\tilde{\gamma}^{n-1})$ 
and $id \not = \tilde{\gamma}$,   we have 
$$0=id+ \tilde{\gamma}+ \cdots +\tilde{\gamma}^{n-1} \in \enn(\alb(\Xc_n^{(i)}))= \enn(\alb(\Xc_n)),$$
which implies that $\prym_{\Xc_n^{(i)}/\Pp^2}=\alb(\Xc_n^{(i)})=\alb(\Xc_n)$ for  $i=1, \ldots, m.$ 
Then, using  (\ref{eq2}), we get a $k$-isogeny of abelian varieties,
\begin{equation}
\label{prym}
\prym_{\Uc_m/\V_m} \sim_k \prod_{i=1}^m \prym_{\Xc_n^{(i)}/\Pp^2} = \alb(\Xc_n)^m.
\end{equation}
Now, let us consider the $1$-cocycle  $a=(a_u) \in \Zc^1(\bar{G}, \aut(\alb(\Xc_n)))$ defined  by $a_{id}=id$ and $a_{\bar{\gamma}^j}=\bar{\tau_*}^j$ where  $\bar{\gamma}^j \in \bar{G}$ and 
$\bar{\tau}_*: \alb(\Xc_n) \rightarrow \alb(\Xc_n)$ is the automorphism induced by
$\bar{\tau}: \Xc_n \rightarrow \Xc_n$.
Then, using Proposition \ref{P1}, we conclude that
$\widetilde{\alb(\Xc_n)}\sim_\KK \alb(\tilde{\Xc}_n)$.
Thus, by  Theorem 1.1 in \cite{ref21},  for $\Xc'=\Uc_m$,  $\Xc=\V_m$, and $\A =\alb(\Xc_n)$, we get:  
\begin{align*}
\widetilde{\alb(\Xc_n)} (\KK)  & \cong  \Hm_{k} (\prym_{\Uc_m/\V_m}, \alb(\Xc_n)) \oplus \alb(\Xc_n)[m](k)\\
& \cong  \Hm_{k} (\alb(\Xc_n)^m , \alb(\Xc_n)) \oplus \alb(\Xc_n)[m](k)\\
& \cong (\enn_{k}(\alb(\Xc_n)))^m    \oplus \alb(\Xc_n)[m](k).
\end{align*}
We denote by $\tilde{Q}_i$  the image of $\tilde{P}_i$ by  the Albanese map 
$\tilde{\alpha}_n: \tilde{\Xc}_n \rightarrow \alb(\tilde{\Xc}_n)$  for  $i=1,\ldots, m$. Then,
by tracing back the above isomorphisms, one can  see that the points 
$\tilde{Q}_1, \ldots, \tilde{Q}_m$ 
form a  subset of independent generators for the Mordell-Weil group $\alb(\tilde{\Xc}_n)(\KK)$. Hence, as $\Z$-modules,  we have 
$$\rk (\widetilde{\alb(\Xc_n)} (\KK) )=\rk(\alb(\tilde{\Xc}_n))(\KK)= m\cdot \rk (\enn_{k}(\alb(\Xc_n))).$$
Therefore, we have completed the proof of Theorem \ref{main1}.

\section{Proofs of Theorems \ref{main2} and  \ref{main3}}
\label{MP2}

In order to prove Theorem \ref{main2}, we examine the ring of endomorphisms of 
$\alb(\Xc_n)$ for $3 \leq n \leq 12$ and $n\not =7, 9, 11$ case by case.
We suppose that $k$ contains $\Q(\zeta_n)$ and  there exists a $k$-rational point on 
$ X_n$ and hence on   $\Xc_n$. Furthermore, we assume that \eqref{ass1} holds for $\alb(\Xc_n)$.
By Theorem \ref{main1}, 
$$\widetilde{\alb(\Xc_n)} (\KK)   \cong  \big (\enn_k( \alb(\Xc_n))\big)^m \oplus \alb(\Xc_n)[m](k).$$
Since  $\alb(\Xc_n)[m](k)$ is a trivial or a finite group, so it does not distribute to the rank of 
$\widetilde{\alb(\Xc_n)} (\KK) $  as a $\Z$-module, but   the rank of $\enn_k( \alb(\Xc_n))$ does. 
We recall that $d_n$ is the dimension of  Albanese variety $\alb(\Xc_n).$

First,  we let   $n=3$ and assume that $\alpha(\Xc_3)$ is a surface. Then, by the part (ii)  of Proposition \ref{p2}, we have $\alb(\Xc_3) \cong \Ee_{\rho}^{d_3}$ which implies that 
$$\displaystyle \enn_k(\alb(\Xc_3)) \otimes \Q =M_{d_3}\left( \enn_k(\Ee_{\rho}) \otimes \Q \right) ,$$ where
$\Ee_{\rho}: y^2=x^3+1$ and $M_*(\cdot)$ denotes the ring of  $(*, *)$-matrices with entries in its origin $(\cdot)$.
Since we have assumed that $\rho=\zeta_3 \in k$, we have  $\enn_k(\Ee_{\rho})\cong \Z[\rho]$ which is of rank 2 as a $\Z$-module.  Thus, $M_{d_3}(\enn_k(\Ee_{\rho}) \otimes \Q )$ has rank $d_3^2$  as a $\Z[\rho]$-module and
hence it is of rank $2 d_3^2$  as a $\Z$-module.
Therefore,  we conclude that
$$\rk(\widetilde{\alb(\Xc_3)} (\KK))=m\cdot \rk(\enn_k(\alb(\Xc_3)))= 2m d_3^2. $$
Similar arguments work for the case $n=4$ and $6$, where in the former case we have to use the fact that  
$\enn_k(\Ee_{i}) \cong \Z[i]$.

Second, we consider the case $n=5$. By Proposition \ref{p2} (iv), $d_5$ is an even number and we have
$\alb(\Xc_5)=J(C_1)^{d_5/2}$, where $C_1: y^2=x^5+1$. This   implies that
$\displaystyle \enn_k(\alb(\Xc_5)) \otimes \Q =M_{d_5/2}(\enn_k(J(C_1)) \otimes \Q ).$ 
We note that the Jacobian variety $J(C_1)$ is a simple abelian variety and its endomorphism ring
$\enn_k(J(C_1))$ contains  $\Z[\zeta_5]$ as a $\Z$-submodule of   rank $4$.
Thus, $M_{d_5/2}(\enn_k(J(C_1)) \otimes \Q)$ contains $M_{d_5/2}(\Q(\zeta_5))$.
This gives us that $\enn_k(\alb(\Xc_5))$
has rank at least  $d_5^2$ as a $\Z$-module.
Therefore,  
$$\rk(\widetilde{\alb(\Xc_5)} (\KK))=m\cdot \rk(\enn_k(\alb(\Xc_5)))\geq   m d_5^2. $$
A similar arguments leads to the proof in the  case $n=10$.

Finally, we consider the case $n=8$ and leave  $n=12$ for the reader.
The part (v) of Proposition \ref{p2} implies  that $d_8$ is an even number  and  one has
$\alb(\Xc_8)=J(C_2)^{n_1} \times \Ee_i^{2 n_2}$,
where $C_2$ is the normalization of the projective closure of the affine curve $ y^2=x^5+x$  and $n_1$ and $n_2$ are positive integers such that $n_1+n_2=d_8/2$. Thus, 
$$\displaystyle \enn_k(\alb(\Xc_5)) \otimes \Q =M_{n_1}(\enn_k(J(C_2))) \oplus  M_{2n_2}(\enn_k(\Ee_i)).$$
Since the Jacobian variety $J(C_2)$ splits as the product of the  elliptic curves,
$$\Ee_1:  y^2=x^3+x^2-3x+1, \ \text{and} \ \Ee_2:  y^2=x^3-x^2-3x+1,$$
we have $\enn_k(J(C_2))=\enn_k(\Ee_1)\oplus \enn_k(\Ee_2)$ and hence
$$\displaystyle \enn_k(\alb(\Xc_5)) \otimes \Q =M_{n_1}(\enn_k(\Ee_1)) \oplus M_{n_1}(\enn_k(\Ee_2)) \oplus M_{2n_2}(\enn_k(\Ee_i)).$$
One can check that the endomorphism ring of $\Ee_1$ and $\Ee_2$ contains the rings $\Z[\sqrt{-2}]$ and $\Z[\sqrt{-3}]$, respectively, which are 
of rank $2$ as $\Z$-modules.  This means that the rank of 
$\enn_k(J(C_2))$ is at least $ 4$ as a $\Z$-module.
Thus,  considering the fact  $\enn_k(\Ee_i)=\Z[i]$ and $n_1+n_2=d_8/2$ , one can conclude that
$\enn_k(\alb(\Xc_5)) $ is a $\Z$-module of rank at least $4(n_1^2+n_2^2)=d_8^2-8 n_1 n_2$. Hence,
$$\rk(\widetilde{\alb(\Xc_8)} (\KK))\geq m\cdot (d_8^2-8 n_1 n_2). $$
We refer the reader to \cite{ref17} for  more details on the above assertions on the elliptic curves $\Ee_i$'s, the Jacobians  $J(C_i)$'s, for $i=1,2$, and their endomorphism rings.
Therefore, we have completed the proof of Theorem \ref{main2} as desired.

To prove Theorem \ref {main3}, we note that 
the Albanese image $\alpha (\Xc_6)$ is a surface by Theorem \ref{kuliko}. 
Thus $\alb(\Xc_6)\cong \Ee^{d_6}_\rho$ by  Proposition \ref{p2} (ii). 
Then, applying Theorem \ref{main2} leads to the equality in the first part of  \ref{main3},
$$ \label{lteq}
\rk(\widetilde{\alb(\Xc_6)} (\KK)) = \rk(\alb(\tilde{\Xc}_6)(\KK))= 2 m d_6. $$
If we assume that  $\Xc_6$ is associated to $\Cc$, the dual of a smooth cubic $\Ee$ or 
the projective model of the affine curve  $$C: f(x,y)=x^3 -3 xy (y^3-8) + 2(y^6 -20 y^3 -8)=0,$$  then
the above equality  and Proposition \ref{p3} proves 
the last assertion of Theorem \ref {main3}.

\end{document}